\documentclass{article}
\usepackage{amsmath,amsfonts,amssymb,amsthm}
\usepackage{graphicx}

\usepackage[pdf]{pstricks}
\usepackage{pst-plot}

\DeclareRobustCommand{\stirling1}{\genfrac[]{0pt}{}}

\newtheorem{proposition}{Proposition}[section]
\newtheorem{theorem}[proposition]{Theorem}
\newtheorem{lemma}[proposition]{Lemma}

\begin{document}

\renewcommand{\thefootnote}{\fnsymbol{footnote}}
\author{{\sc Norbert Henze\footnotemark[1] and G\"unter Last\footnotemark[2]}}
\footnotetext[1]{norbert.henze@kit.edu,  
Karlsruhe Institute of Technology, Institute of Stochastics,
76131 Karlsruhe, Germany. }
\footnotetext[2]{guenter.last@kit.edu,  
Karlsruhe Institute of Technology, Institute of Stochastics,
76131 Karlsruhe, Germany. }

\title{Absent-minded passengers}

\maketitle

\begin{abstract}
\noindent
Passengers board a fully booked airplane in order. The first
passenger picks one of the seats at random.
Each subsequent passenger takes his or her assigned seat if available, otherwise takes
one of the remaining seats at random. It is well known that the last passenger obtains
 her own seat with probability $1/2$. We study the distribution of the number
of incorrectly seated passengers, and we also discuss the case of
several absent-minded passengers.
\end{abstract}


\vspace{0.1cm}

\section{Introduction.}

The following problem is posed and solved in \cite[p.\
  177]{Bolobas2006}, and \cite[p.\ 35]{Winkler2004}.  An airplane with
  $n\ge 2$ passengers is fully booked.  Passengers are boarding in
  chronological order, according to the numbers on their boarding
  passes. The first passenger loses his boarding pass and picks one of
  the seats at random.  Each subsequent passenger takes his or her
  seat if available, otherwise takes one of the remaining seats at
  random. The following intuitive argument shows that the last
  passenger will sit in her own seat number $n$ with probability $1/2$.
  The first absent-minded passenger chooses either his assigned seat
  number 1, or seat number $n$, or one of the seats numbered
  $2,\ldots,n-1$. In the first case, the last passenger will sit in
  her seat, and in the second case she has to sit down in seat number 1.
  In each of the remaining cases, the role of the absent-minded
  passenger will be taken over by the first passenger who finds his
  seat occupied.  In the end, everything boils down to a toss of a
  fair coin that decides whether one of the passengers chooses seat
  number 1 or seat number $n$.  The same argument shows that the last but
  one passenger will sit in his assigned seat with probability
  $2/3$. Likewise, the probability is $3/4$ for the antepenultimate
  passenger to sit in his assigned seat, etc.
In this article we study
the distribution of the number of incorrectly seated passengers
in the more general case where the first $k$ passengers are absent-minded.

We find it both useful and informative to first focus on the case $k=1$.
Even though the results are then basically known
(see \cite{Bolobas2006,Misra2008,Nigussie2014,Winkler2004}), we take
a partly new approach based on probabilistic arguments.
As long as the first seat is not occupied, there is, at each
stage of the boarding procedure, exactly
one incorrectly seated passenger in the remaining seats.
The (suitably defined) conditional distribution of
the number of this seat is uniform; see Lemma \ref{l2.2}.
This fact easily implies a certain independence property
(Proposition \ref{p2.3}) and then in turn
an explicit formula for the distribution
of the number of incorrectly seated passengers and
a central limit theorem. In Section 3 we extend these
arguments to the case where the first $k$ passengers are absent-minded.
While there is still an independence property
(Theorem \ref{p3.3}), the full distribution of the
number of incorrectly seated passengers seems to be difficult to obtain.
Still we have been able to derive explicit formulae for
mean and variance along with a central limit theorem.
Except for the mean, these results are new.

\section{One absent-minded passenger.}

First we describe our model in a neutral, more abstract language. Consider
$n$ balls numbered $1,\ldots,n$, to be placed in chronological order
in a box with $n$ numbered pits. The first ball is placed at random.
The second ball is placed in pit 2, provided this pit is empty. Otherwise
it is placed at random in one of the remaining $n-1$ pits.
This procedure continues until the $n$th ball is placed.

In this section, we derive the distribution of the random number $W_n$ (say) of balls that do not
meet their assigned pit. Moreover, we show that the limit distribution of $W_n$, when suitably normalized,
is standard normal as $n \to \infty$. Most parts of this article are based on elementary probability theory
as can be found in numerous textbooks (see, e.g., \cite{Chung2003} or \cite{Grimmet2014}).

In what follows, let $X_1$ denote the
number of the pit taken by the first ball. By definition, $X_1$ is uniformly distributed
on $\{1,\ldots,n\}$. For each  $i \in \{2,\ldots,n-1\}$, consider the state of the box
after the placement of the $i$th ball. If pit 1 is occupied, then each subsequently placed
ball matches the number of its assigned pit. In this case we define  $X_i:=1$.
If pit 1 is empty, then each of the pits $2,\ldots,i$ is occupied, as well as exactly one
further pit having number $j$ (say), where $j>i$.
 In this case we put $X_i:=j$. Finally, we define $X_n:=1$.

 To illustrate the notation, we consider the case $n=9$. If the first
 ball takes pit 4, the fourth ball takes pit 8, and ball 8 takes pit 1,
 then the balls numbered 2,3,5,6,7, and 9 find their assigned pits, and
 we have $X_1= X_2=X_3 =4$, $X_4=X_5=X_6=X_7=8$, and $X_8=X_9=1$.

 If we denote by $N_j$ the number of the pit taken by the $j$th ball,
 where $j=1,\ldots,n$, then $(N_1,\ldots,N_n)$ is a 
random permutation of $(1,\ldots,n)$, and we have
$X_i=1$ if $\min(N_1,\ldots,N_i) =1$ and $X_i = \max(N_1,\ldots,N_i)$ otherwise.
In
other words, $X_i=1$ if and only if pit 1 is occupied after the
placement of the $i$th ball. Otherwise, $X_i$ is the largest number
of the occupied pits at that stage. Thus, for example, the permutation
(2,1,3,4) of (1,2,3,4) arises if the first ball takes pit 2 and the second ball (that finds its pit occupied)
chooses pit 1. Clearly, balls 3 and 4 then find their assigned pits.

Table 1 shows the feasible
permutations $(N_1,\ldots, N_n)$ for the case $n=4$, together with
their respective probabilities.  For each permutation, the last column
displays the number of ``incorrectly placed balls," i.e., balls that do not
meet their assigned pit.  \vspace*{3mm}
\begin{center}
\begin{tabular}{|cccc|c|c|}\hline \label{tab1}
$N_1$ & $N_2$ & $N_3$ & $N_4$ & prob & $W_4$\\ \hline
 1 & 2 & 3 & 4 & $1/4$ & 0\\
 2 & 1 & 3 & 4 & $1/12$ & 2\\
 2 & 3 & 1 & 4 & $1/24$ & 3 \\
 2 & 3 & 4 & 1 & $1/24$ & 4 \\
 2 & 4 & 3 & 1 & $1/12$ & 3\\
 3 & 2 & 1 & 4 & $1/8$ & 2\\
 3 & 2 & 4 & 1 & $1/8$ & 3\\
 4 & 2 & 3 & 1 & $1/4$ & 2\\ \hline
\end{tabular}
\end{center}
\begin{center}
Table 1: The feasible permutations $(N_1,N_2,N_3,N_4)$ and their probabilities.
\end{center}

\bigskip
The model has the following simple but crucial symmetry property.

\begin{lemma}\label{l2.1} Let $i\in\{1,\ldots,n\}$. Then $X_i$ is
uniformly distributed on $\{1,i\!+\!1,\ldots,n\}$.
\end{lemma}
\begin{proof} We use induction on $i$. For $i=1$ and $i=n$
the assertion is true.  Assume that it holds for some
$i\in\{2,\ldots,n-1\}$ and let $j\in\{1,i+2,\ldots,n\}$.
Conditioning on $X_i$ and noting that $X_{i+1}=j$ implies
either $X_i=j$ or $X_i = i+1$, we obtain
\begin{align*}
{\mathbb{P}}(X_{i+1}=j)&= {\mathbb{P}}(X_i=j) {\mathbb{P}}(X_{i+1}=j|X_i=j)\\
& \qquad  + {\mathbb{P}}(X_i= i+1){\mathbb{P}}(X_{i+1}=j|X_i =i+1)\\
&= \frac{1}{n-i+1} \cdot 1 + \frac{1}{n-i+1} \cdot \frac{1}{n-i}\\
&=\frac{1}{n-i},
\end{align*}
and the proof is finished.
\end{proof}

For $i\in\{1,\ldots,n\}$, let $A_i$ denote the event
that the $i$th ball is placed in its assigned pit.
Since  $A_i=\{X_{i-1}\ne i\}$
for $i\ge 2$, Lemma \ref{l2.1} yields
\begin{align}\label{e2.3}
{\mathbb{P}}(A_i)=\frac{n-i+1}{n-i+2},\quad i=2,\ldots,n;
\end{align}
see also \cite{Nigussie2014}.  Thus, in particular,
${\mathbb{P}}(A_n)=1/2$, which is the solution to the initial problem
given in \cite{Bolobas2006} and \cite{Winkler2004}.

Incidentally, it follows by induction that there are
  $2^{n-1}$ feasible permutations. In fact, the last ball can only
  meet either its assigned pit or pit $1$, and each of these cases
  gives rise to the same number of feasible permutations
  $(N_1,\ldots,N_n)$. Indeed, if
  $N_n=n$, then the admissible assigments of balls $1,\ldots,n-1$ to
  pits $1,\ldots,n-1$ coincide with the feasible permutations in the
  case of $n-1$ pits. On the other hand, if $N_n=1$, then we are in the
  situation of $n-1$ balls and $n-1$ pits numbered
  $n,2,\ldots,n-1$. Here, pit $n$ is assigned to the first ball to be
  placed, and this ball is ``absent-minded"; see Table 1 for the case $n=4$.
Hence, the case
  $N_n=1$ also gives rise to $2^{n-1}$ feasible permutations.

Writing ${\bf 1}_A$ for the indicator function of an event $A$, let
$$
C_n:=\sum^n_{i=1}{\mathbf 1}_{A_i}
$$
denote the number of correctly placed balls. Since ${\mathbb{P}}(A_1) = 1/n$,
\eqref{e2.3} gives
\begin{align*}
{\mathbb{E}}(C_n) =\frac{1}{n}+\sum^n_{i=2}\bigg(1-\frac{1}{n-i+2}\bigg)
=\frac{1}{n}+\sum^{n-2}_{j=0}\bigg(1-\frac{1}{j+2}\bigg),
\end{align*}
so that
\begin{align}\label{e2.11}
{\mathbb{E}}(C_n)=n-H_{n-1},
\end{align}
where $H_m:=1+1/2+\cdots+1/m$, $m\in {\mathbb N}$, is the $m$th {\em harmonic number}.

Below we will show that the events $A_2,\ldots,A_n$ are independent,
a result that might come as a surprise.
To do so we need the following lemma.

\begin{lemma}\label{l2.2} Suppose that $r\in\{1,\ldots,n-2\}$
and $2\le i_1<\cdots<i_r\le n-1$. Let $i\in\{i_r+1,\ldots,n\}$.
Then, under the condition $A_{i_1}\cap\cdots\cap A_{i_r}$, the
random variable $X_i$ has a uniform distribution on $\{1,i+1,\ldots,n\}$.
\end{lemma}
\begin{proof} For ease of notation we prove the result in the case $r=1$, the
general case being 
analogous. Let $m:=i_1$. We need to show
that ${\mathbb{P}}(X_i=j| A_m)=1/(n-i+1)$ for each $i\in\{m+1,\ldots,n\}$
and each $j\in\{1,i+1,\ldots,n\}$. We proceed as in the proof of
Lemma \ref{l2.1} and use induction on $i$.

Let $j\in\{1,m+2,\ldots,n\}$. By the definition of our model and
Lemma \ref{l2.1},
\begin{align*}
\mathbb{P}(X_{m+1}=j,X_{m-1}\ne m)&=\mathbb{P}(X_{m-1}=j)+\mathbb{P}(X_{m-1}= m+1, X_{m+1}=j)\\
&=\frac{1}{n-m+2}+\mathbb{P}(X_{m-1}= m+1)\cdot\frac{1}{n-m}\\
&=\frac{1}{n-m+2}\cdot\frac{n-m+1}{n-m}.
\end{align*}
Using Lemma \ref{l2.1} again, we hence obtain
\begin{align*}
{\mathbb{P}}(X_{m+1}=j|A_m)
&=\frac{{\mathbb{P}}(X_{m+1}=j,X_{m-1}\ne m)}{{\mathbb{P}}(X_{m-1}\ne m)}=\frac{1}{n-m}.
\end{align*}
Assume now the assertion is true for each $i\in\{m+1,\ldots,n-1\}$
and take some $j\in\{1,i+2,\ldots,n\}$. Then
\begin{align*}
{\mathbb{P}}(X_{i+1}=j| A_m)& =  {\mathbb{P}}(X_i=j| A_m)+{\mathbb{P}}(X_{i+1}=j,X_i=i+1| A_m)\\
& =  \frac{1}{n-i+1} + {\mathbb{P}}(X_{i+1}=j|\{X_i=i+1\}\cap A_m){\mathbb{P}}(X_i=i|A_m)\\
& =  \frac{1}{n-i+1} +\frac{1}{n-i}\cdot\frac{1}{n-i+1}=\frac{1}{n-i}.
\end{align*}
This finishes the induction and hence the proof.
\end{proof}

The following independence property can be found (with a different proof)
as Theorem 3.5 in \cite{Lengyel2010}.

\begin{theorem}\label{p2.3}
The events $A_2,\ldots,A_n$ are independent.
\end{theorem}
\begin{proof} Let $r\in\{1,\ldots,n-1\}$
and $2\le i_1<\cdots<i_r\le n-1$. By Lemma \ref{l2.2},
\begin{align*}
{\mathbb{P}}(A_{i_1}\cap\cdots \cap A_{i_r})
&={\mathbb{P}}(A_{i_r}|A_{i_1}\cap\cdots \cap A_{i_{r-1}}){\mathbb{P}}(A_{i_1}\cap\cdots \cap A_{i_{r-1}})\\
&=\frac{n-i_r+1}{n-i_r+2}\cdot{\mathbb{P}}(A_{i_1}\cap\cdots \cap A_{i_{r-1}}).
\end{align*}
Repeating this reasoning $r-1$ times and using Lemma \ref{l2.1} in the final step,
we obtain
\begin{align*}
{\mathbb{P}}(A_{i_1}\cap\cdots \cap A_{i_r})
=\frac{n-i_r+1}{n-i_r+2}\cdots \frac{n-i_1+1}{n-i_1+2},
\end{align*}
as required.
\end{proof}


Writing $A^c$ for the complement of an event $A$, let
$$
W_n := n- C_n = \sum_{i=1}^n {\mathbf 1}_{A_i^c}
$$
denote the number of balls that do not meet their assigned pit. From
Table 1, the probability distribution of $W_4$ is given by
$$
{\mathbb{P}}(W_4 =0) = \frac{6}{24}, \ {\mathbb{P}}(W_4=2) =
\frac{11}{24}, \ {\mathbb{P}}(W_4=3) = \frac{6}{24}, \
{\mathbb{P}}(W_4=4) = \frac{1}{24}.
$$
Here, the numbers $6=\stirling1{4}{1}$, $11=\stirling1{4}{2}$,
$6=\stirling1{4}{3}$, and $1=\stirling1{4}{4}$ figuring in the
respective numerators are {\em Stirling numbers of the first kind}.
For general $n\in\mathbb{N}$
and $j\in\{1,\ldots,n\}$, the Stirling number
$\stirling1{n}{j}$ is the number of permutations of $\{1,\ldots,n\}$ having
exactly $j$ cycles; see, e.g., \cite[Section 6.1]{Graham1994}. Notice that $\stirling1{n}{1} = (n-1)!$.
The next result shows that the occurrence of the Stirling numbers in the special case $n=4$ is
not a lucky coincidence; see \cite[Theorem 3.6]{Lengyel2010}.

\begin{theorem}\label{p2.5}
The distribution of $W_n$ is given by
$$
{\mathbb{P}}(W_n=0) = \frac{\stirling1{n}{1}}{n!},
\quad {\mathbb{P}}(W_n=j) = \frac{\stirling1{n}{j}}{n!}, \quad j=2,\ldots, n.
$$
\end{theorem}
\begin{proof} For $i\in\{2,\ldots,n\}$ define $B_i:=A_{n-i+2}$.
Then
\begin{equation}\label{reco}
W_n = R_n - {\mathbf 1}_{A_1},
\end{equation}
where
$
R_n := 1 + \sum_{i=2}^n {\mathbf 1}_{B_i}.
$
By Theorem \ref{p2.3}, the events $B_2, \ldots, B_n$ are
independent, and we have ${\mathbb{P}}(B_i) = 1/i$, $i\in\{2,\ldots,n\}$.
We thus have, for each $n \ge 1$ and each $j \in \{2,\ldots,n\}$,
\begin{align}\label{e2.5}\notag
{\mathbb{P}}(R_{n+1}=j)&={\mathbb{P}}(\{R_{n+1}=j\}\cap B^c_{n+1})+{\mathbb{P}}(\{R_{n+1}=j\}\cap B_{n+1})\\
&=\frac{n}{n+1}{\mathbb{P}}(R_n=j)+\frac{1}{n+1}{\mathbb{P}}(R_n=j-1).
\end{align}
Hence, the {\em generating function} $f_n(z):={\mathbb{E}}(z^{R_n})$, $z\in[0,1]$, of $R_n$
satisfies the recursion
\begin{align*}
f_{n+1}(z)=\frac{n+z}{n+1}f_n(z).
\end{align*}
Since $f_1(z)=z$, it follows that
\begin{align*}
f_n(z)=\frac{z(z+1)\cdots (z+n-1)}{n!},\quad z\in[0,1].
\end{align*}
Up to the factor $1/n!$ this expression is the generating function of
the Stirling numbers of the first kind (see \cite[p.\ 252]{Graham1994}),
so that
\begin{align}\label{records}
{\mathbb{P}}(R_n =j) = \frac{\stirling1{n}{j}}{n!}, \quad j=1,\ldots, n.
\end{align}
Since, by \eqref{reco}, $R_n = W_n$ on
the event $A_1^c$ and $W_n=0$ on the event $A_1$ (since
$A_1=A_1\cap A_2 \cap \cdots \cap A_n)$, the result follows.
\end{proof}


Formula \eqref{records} gives  the distribution
of the number of cycles in a purely random permutation of
$\{1,2,\ldots,n$\} (see, e.g.,\ \cite[Section 6.2]{Blom1994}),
as well as the distribution of the number of records
in a sequence of independent and identically distributed
continuous random variables (see,
e.g., \cite[Section 9.5]{Blom1994}).
In both cases, the reason is a  recursion of the type \eqref{e2.5}.

Figure 1 shows a bar chart of the distribution of
$W_{100}$. 
The most probable value $(= 0.2112$) for the number of
incorrectly seated passengers is attained at $\ell=5$.  
Since $|W_n-R_n| \le 1$, the Lindeberg--Feller central limit theorem and
Sluzki's lemma show that the limit distribution of
$(W_n - \log n)/\sqrt{\log n}$ as $n \to \infty$ is standard normal; see, e.g.,
\cite[p.\ 383]{Billingsley1986}, and, in particular, Example 27.3.

\vspace*{5mm}

\begin{center}
\psset{xunit=.6cm,yunit=15cm}
\psset{plotpoints=300}
\begin{pspicture}(-.7,-.03)(15.5,.25)
\psaxes[ticks=none,labels=none]{->}(0,0)(-.7,0)(15,.25)
\psline(-.2,.05)(0,.05)
\psline(-.2,.1)(0,.1)
\psline(-.2,.15)(0,.15)
\psline(-.2,.2)(0,.2)
\rput(-.8,.05){$.05$}
\rput(-.8,.10){$.10$}
\rput(-.8,.15){$.15$}
\rput(-.8,.20){$.20$}
\psline(1,-.011)(1,0)
\psline(2,-.011)(2,0)
\psline(3,-.011)(3,0)
\psline(4,-.011)(4,0)
\psline(5,-.011)(5,0)
\psline(6,-.011)(6,0)
\psline(7,-.011)(7,0)
\psline(8,-.011)(8,0)
\psline(9,-.011)(9,0)
\psline(10,-.011)(10,0)
\psline(11,-.011)(11,0)
\psline(12,-.011)(12,0)
\psline(13,-.011)(13,0)
\psline(14,-.011)(14,0)
\rput(1,-.035){$0$}
\rput(2,-.035){$1$}
\rput(3,-.035){$2$}
\rput(4,-.035){$3$}
\rput(5,-.035){$4$}
\rput(6,-.035){$5$}
\rput(7,-.035){$6$}
\rput(8,-.035){$7$}
\rput(9,-.035){$8$}
\rput(10,-.035){$9$}
\rput(11,-.035){$10$}
\rput(12,-.035){$11$}
\rput(13,-.035){$13$}
\rput(14,-.035){$14$}
\psline[linewidth=2mm](1,0)(1,.01)
\psline[linewidth=2mm](3,0)(3,.05177)
\psline[linewidth=2mm](4,0)(4,.1258)
\psline[linewidth=2mm](5,0)(5,.19298)
\psline[linewidth=2mm](6,0)(6,.2112)
\psline[linewidth=2mm](7,0)(7,.1767)
\psline[linewidth=2mm](8,0)(8,.1181)
\psline[linewidth=2mm](9,0)(9,.06510)
\psline[linewidth=2mm](10,0)(10,.03024)
\psline[linewidth=2mm](11,0)(11,.01206)
\psline[linewidth=2mm](12,0)(12,.00418)
\psline[linewidth=2mm](13,0)(13,.001277)
\psline[linewidth=2mm](14,0)(14,.000346)
  \rput(2.2,.26){${\mathbb{P}}(W_{100}=\ell)$}
  \rput(15.3,-.02){$\ell$}
 \end{pspicture}
\end{center}
\begin{center}
Figure 1: Bar chart of the distribution of $W_{100}$.
\end{center}

\section{Several absent-minded passengers.}

Generalizing the problem discussed in \cite{Bolobas2006}, \cite{Nigussie2014}, and
\cite{Winkler2004}, we now consider the case when the first $k$
passengers are absent-minded, where $k \in \{1,\ldots,n-1\}$.  If
these passengers take their seats completely at random and the
subsequent passengers board the airplane according to the rule stated
in the introduction, what is the probability that the last passenger
finds his or her seat available? We will see that the answer is
$1/(k+1)$. Moreover, we will compute the expectation and variance of
and a central limit theorem for
the number of incorrectly seated passengers.

In a more neutral formulation, suppose again that we have $n$ balls to
be allocated to $n$ pits in a box.  Let $k\in\{1,\ldots,n-1\}$, and
assume that the first $k$ balls are distributed at random, each subset
of size $k$ of all $n$ pits having the same probability of being
chosen.  The $(k+1)$st ball takes pit $k+1$ provided that pit is
available.  Otherwise it chooses its position at random. This
procedure continues until the last ball is placed.

Given $i\in\{k,\ldots,n\}$, we define a random subset $Z_i$ of
$\{1,\ldots,k\}\cup \{i\!+\!1,\ldots,n\}$ as follows.  Consider the state
of the box after the placement of the $i$th ball. If $i=k$, let $Z_i$
be the set of pits occupied by the ``absent-minded balls."
If $i>k$, then all pits from $\{k+1,\ldots,i\}$ are occupied, and we write
$Z_i$ for the set of all {\em other} occupied pits.
Then $Z_i$ has $k$ elements, i.e., we have $|Z_i|=k$.

To illustrate the new notation, consider the case $n=10$, $k=3$ and the placement
\begin{center}
\begin{tabular}{c|cccccccccc}
ball no.  & 1 & 2& 3& 4& 5& 6& 7& 8& 9 & 10\\ \hline
pit no. & 7 & 3 & 9 & 4& 5& 6& 1& 8 & 10 & 2
\end{tabular} .
\end{center}
Thus, ball 1 takes pit 7, ball 2 chooses pit 3, ball 3 meets pit 9, etc.
In this case, we have $Z_3 = Z_4 = Z_5 = Z_6 = \{3,7,9\}$, $Z_7=Z_8 = \{1,3,9\}$, $Z_9=\{1,3,10\}$, and
$Z_{10} = \{1,2,3\}$.

Let $\mathcal{P}_i$ denote the system of all
sets $A\subset\{1,\ldots,k\}\cup\{i+1,\ldots,n\}$ with
$|A|=k$. The next result generalizes Lemma \ref{l2.1}.

\begin{lemma}\label{l3.1} Let $i\in\{k,\ldots,n\}$. Then the distribution of $Z_i$ is
uniform on $\mathcal{P}_i$.
\end{lemma}
\begin{proof} We proceed by induction on $i$. For $i=k$ the
assertion holds. Suppose it is true for some $i\in\{k,\ldots,n-1\}$
and let $A\in\mathcal{P}_{i+1}$. Since $\{Z_i=A\} \subset \{Z_{i+1} =A\}$
and $Z_i\ne Z_{i+1}$ if and only if $i+1\in Z_i$, we have
\begin{align*}
{\mathbb P}(Z_{i+1}=A)&={\mathbb P}(Z_{i+1}=A = Z_i)+{\mathbb P}(Z_{i+1}=A,Z_i\neq A)\\
&={\mathbb P}(Z_i=A)+\sum_{j\in A}{\mathbb P}(Z_i=(A\setminus\{j\})\cup\{i+1\}) \cdot \frac{1}{n-i}\\
&=\binom{n-i+k}{k}^{-1}\bigg(1+\frac{k}{n-i}\bigg)
=\binom{n-i+k}{k}^{-1}\frac{n-i+k}{n-i}\\
&=\binom{n-(i+1)+k}{k}^{-1}.
\end{align*}
This finishes the induction and hence the proof.
\end{proof}


For $i\in\{1,\ldots,n\}$ again let $A_i$ denote the event
that the $i$th ball is placed in its pit.
Clearly,
\begin{align}\label{e3.0}
{\mathbb P}(A_i)= \frac{1}{n}, \quad i =1,\ldots,k.
\end{align}
Since  $A_i=\{i\notin Z_{i-1}\}$
for $i\ge k+1$, Lemma \ref{l3.1} entails
\begin{align}\label{e3.3}
{\mathbb P}(A_i)=\frac{n-i+1}{n-i+k+1},\quad i=k+1,\ldots,n,
\end{align}
and in particular ${\mathbb P}(A_n)=1/(k+1)$.

We note in passing that there are $k!(k+1)^{n-k}$
feasible permutations; see \cite[Theorem 3.9]{Lengyel2010}. The proof uses the
so-called canonical cycle representation of a permutation.
An alternative argument can be based on induction, similarly to the
case $k=1$ discussed in the paragraph after equation \eqref{e2.3}.

A conceptual proof of Lemma \ref{l3.1} uses the fact
that, after the placement of the $i$th ball, where $i>k$, each of the
places $k+1, \ldots, i$ is occupied. Therefore, the {\em other} $k$
occupied places are among the $n-i+k$ places numbered $1,\ldots,k$ and
$i+1,\ldots,n$.  By symmetry, each choice of $k$ (occupied) places
from these $n-i+k$ places has the same probability. In particular, the
probability that the ($i+1$)st ball finds its assigned pit empty
equals $(n-i)/(n-i+k)$.

Writing $W_{n,k} = \sum_{i=1}^n {\mathbf 1}_{A_i^c}$ for the number of balls that do not meet their
assigned pit, we now have
\begin{align*}
{\mathbb E}(W_{n,k})
& = \sum_{i=1}^n {\mathbb P}(A_i^c) = k \bigg(1-\frac{1}{n}\bigg)+\sum_{i=k+1}^n \frac{k}{n-i+k+1}\\
& = k \left(1 + H_{n-1} - H_k\right).
\end{align*}

The following result can be proved with the help of an analog of Lemma \ref{l2.2}.
We leave this to the reader.

\begin{theorem}\label{p3.3}
The events $A_{k+1},\ldots,A_n$ are independent.
\end{theorem}


We now derive a formula for the variance of $W_{n,k}$. Since
$W_{n,k} + C_{n,k} = n$, where
$C_{n,k} = \sum_{i=1}^n {\mathbf 1}_{A_i}$ is the number of balls that
meet their assigned pit, we have
${\mathbb V}(W_{n,k}) = {\mathbb V}(C_{n,k})$. Now, $C_{n,k}$ is a sum
of indicator random variables, whence
\begin{align}\label{e3.3a} {\mathbb V}(C_{n,k}) = \sum_{i=1}^n
  {\mathbb P}(A_i)\left(1- {\mathbb P}(A_i)\right) + 2 \sum_{1 \le i <
    j \le n} \left({\mathbb P}(A_i \cap A_j) - {\mathbb
      P}(A_i){\mathbb P}(A_j)\right).
\end{align}
Using (\ref{e3.0}) and (\ref{e3.3}), the single sum equals
$$
\frac{k(n-1)}{n^2} + k(H_n-H_k) - k^2 \sum_{\ell =k+1}^n \frac{1}{\ell^2}.
$$
In view of Theorem \ref{p3.3}, only pairs $(i,j)$ satisfying
either $1 \le i < j \le k$ or $1 \le i \le k < j \le n$ make a
nonzero contribution to the double sum figuring in (\ref{e3.3a}).
Since ${\mathbb P}(A_i\cap A_j) = 1/(n(n-1))$ if $1 \le i < j \le k$,
it follows that
$$
\sum_{1 \le i < j \le k} \left({\mathbb P}(A_i \cap A_j) - {\mathbb
    P}(A_i){\mathbb P}(A_j)\right) = \binom{k}{2} \frac{1}{n^2(n-1)}.
$$
If $i \in \{1,\ldots,k\}$ and $j \in \{k+1,\ldots,n\}$, we write
${\mathbb P}(A_i \cap A_j) = {\mathbb P}(A_i){\mathbb P}(A_j|A_i)$.
Under the condition $A_i$, the situation is that of a box containing
$n-1$ pits numbered $1,\ldots,i-1,i+1,\ldots,n$, and the balls
numbered $1,\ldots,i-1,i+1,\ldots,k$ are distributed at random. By
relabelling each ball $j$, where $j>i$, with $j-1$, we can use formula
(\ref{e3.3}) with $n,k$, and $i$ replaced with $n-1,k-1$, and $j-1$,
respectively, and obtain
$$
{\mathbb P}(A_i \cap A_j) = \frac{1}{n} \cdot \frac{n-j+1}{n-j+k}, \quad j=k+1,\ldots,n.
$$
Consequently,
$$
\sum_{i=1}^k \sum_{j=k+1}^n \! \! \left({\mathbb P}(A_i\! \cap \! A_j) - {\mathbb P}(A_i){\mathbb P}(A_j)\right)
= \sum_{i=1}^k \sum_{j=k+1}^n \frac{1}{n} \left(
\frac{n\! -\! j\! +\! 1}{n\! -\! j \! + \! k} - \frac{n \! - \! j \! + \! 1}{n \! - \! j \! + \! k \! + \! 1} \right),
$$
and some algebra shows that this expression equals $\frac{k}{n}\left(H_{n-1}-H_{k-1} - 1 + k/n\right)$.
Putting everything together, straightforward calculations give
$$
{\mathbb V}(W_{n,k})
= k\Bigg[ \frac{2(1\! -\! n\! +\! kn) \! -\! n^2\! -\! k}{n^2(n-1)} + \frac{2}{nk} +
\left(\! 1\! +\! \frac{2}{n}\right)(H_n\! -\! H_k)
- k \! \sum_{\ell =k+1}^n \! \frac{1}{\ell^2} \Bigg ].
$$

We have not been able to find a closed-form expression for the distribution of $W_{n,k}$ if $k \ge 2$.
The asymptotic distribution of $W_{n,k}$ as $n \to \infty$, however, is available.
To this end, writing $a_n \sim b_n$ 
if $a_n/b_n\to 1$ as $n\to\infty$,
and using $H_n \sim \log n$, it follows
that ${\mathbb E}(W_{n,k}) \sim k \log n$ and
${\mathbb V}(W_{n,k}) \sim k \log n$. By the Lindeberg--Feller central
limit theorem, the random variable
$\sum_{i=k+1}^n {\mathbf 1}_{A_i^c}$, after standardization, has a
standard normal limit as $n \to \infty$ (see \cite[p.\ 383]{Billingsley1986}.
 Since $|W_{n,k} - \sum_{i=k+1}^n {\mathbf 1}_{A_i^c}| \le k$,
Sluzki's lemma shows that the limit distribution of
$(W_{n,k} - k \log n)/\sqrt{k \log n}$ as $n \to \infty$ is standard
normal.

\bigskip

\noindent
{\bf Acknowledgment:}
  The authors wish to thank Nicole B\"auerle for drawing their attention
  to the article \cite{Lengyel2010}.

\end{document}